\documentclass{amsart}
\usepackage{amssymb, latexsym}

\newcommand{\beqa}{\begin{eqnarray*}}
\newcommand{\eeqa}{\end{eqnarray*}}
\newcommand{\beqn}{\begin{eqnarray}}
\newcommand{\eeqn}{\end{eqnarray}}

\newcommand{\C}{\mathbb C}

\newcommand{\R}{\mathbb R}

\newcommand{\N}{\mathbb N}
\newcommand{\M}{\mathbb M}

\newcommand{\mcB}{\mathcal B}

\newcommand{\mcE}{\mathcal E}

\newcommand{\mcZ}{\mathcal Z}

\newcounter{cnt1}
\newcounter{cnt2}
\newcounter{cnt3}
\newcommand{\blr}{\begin{list}{$($\roman{cnt1}$)$}
 {\usecounter{cnt1} \setlength{\topsep}{0pt}
 \setlength{\itemsep}{0pt}}}
\newcommand{\bla}{\begin{list}{$($\alph{cnt2}$)$}
 {\usecounter{cnt2} \setlength{\topsep}{0pt}
 \setlength{\itemsep}{0pt}}}
\newcommand{\bln}{\begin{list}{$($\arabic{cnt3}$)$}
 {\usecounter{cnt3} \setlength{\topsep}{0pt}
 \setlength{\itemsep}{0pt}}}
\newcommand{\el}{\end{list}}
\newtheorem{thm}{Theorem}[section]

\newtheorem{cor}[thm]{Corollary}

\newtheorem{Def}[thm]{Definition}
\newtheorem{prop}[thm]{Proposition}
\newtheorem{rem}[thm]{Remark}
\newcommand{\Rem}{\begin{rem} \rm}
\newcommand{\bdfn}{\begin{Def} \rm}
\newcommand{\edfn}{\end{Def}}

\newcommand{\ba}{\begin{array}}
\newcommand{\ea}{\end{array}}

\begin{document}
\begin{center}\large{{\bf{  Zachary spaces  $\mcZ^p[\R^{\infty }]$ and  separable Banach spaces}}} 
 
  Hemanta Kalita$^{1}$, 
   Bipan Hazarika$^{2,\ast}$ and Mohsen Rabbani$^3$
\vspace{0.5cm}

$^{1}$Department of Mathematics, Assam Don Bosco University, Sunapur, Assam, India\\
$^{2}$Department of Mathematics, Gauhati University, Guwahati 781014, Assam, India\\
$^{3}$Department of  Mathematics, Sari Branch, Islamic Azad University, Sari, Iran\\
Email:  hemanta30kalita@gmail.com;
 bh\_rgu@yahoo.co.in; bh\_gu@gauhati.ac.in; mrabbani@iausari.ac.ir

\end{center}
\title{}
\author{}
\thanks{\today}

\maketitle
\begin{abstract} 
 We construct Zachary space in $\R^\infty$ and find that this is a Banach space of functions of bounded mean oscillation with order $p, 1\leq p \leq \infty$ containing the function of bounded mean oscillation $BMO[\R_I^\infty]$ as a dense continuos embedding. As an application of $\R_I^\infty$ we construction $\mcB,$ where $\mcB $ is separable Banach space and finally we construct $\mcZ^p[\mcB]$.\\
	 {\bf{Keywords and phrases:}} Dense; Continuous embedding; Banach space.\\
	 {AMS subject classification No:} 26A39, 46B03, 46B20, 46B25.
 
\end{abstract}

\section{Introduction and Preliminaries}
 Gill and Zachary \cite{GZ} introduced  a new theory of Lebesgue measure on $\R^\infty,$ the construction of which is nearly  same as the development of Lebesgue measure on $\R^n.$ This theory can be useful in formulating a new class of spaces which will provide a Banach space  structure for Henstock-Kurzweil (HK) integrable functions (also see \cite{PrezBecerra,etaltin,BH,SnchezPerales,SnchezPeralesTorres,MKRAA,nsrnns} and references therein) .  
 In \cite{GRA}, authors mentioned that the space BMO is not naturally a Banach space and not separable although BMO is similar type of space of $L^p$. In \cite{DC}, Chang and Sadosky proved that the duality BMO space is the Hardy space ($H^1$).
\begin{Def}
  \cite{GZ} We define $ \R_{I}^{n} = \R^n \times I_n.$ If $ T $ is a linear transformation on $ \R^n $ and $ A_n = A \times I_n ,$ we define $ T_I $ on $ \R_{I}^{n} $ by $ T_I[A_n] = T[A] $, we define $ \mcB[\R_{I}^{n}]$ to be the Borel $\sigma-$algebra for $ \R_{I}^{n},$ where the topology for $ \R_{I}^{n} $ is defined via the following class of open sets $ D_n = \{ U \times I_n : U $ is open in $ \R^n\}.$ For any $ A \in \mcB[\R^n], $ we define $\lambda_\infty(A_n) $ on $ \R_{I}^{n} $ by product measure $ \lambda_\infty(A_n) = \lambda_n(A) \times \Pi_{i= n +1 }^{\infty} \lambda_I(I) = \lambda_n(A).$
  \end{Def}
   \begin{thm} \cite{GM}
 $ \lambda_{\infty}(.) $ is a measure on $ \mcB[\R_{I}^{n}] $ is equivalent to n-dimensional Lebesgue measure on $ \R^n.$
 \end{thm}
\begin{cor}
  The measure $ \lambda_{\infty}(.) $ is both translationally and rotationally invariant on $ ( \R_{I}^{n}, \mcB[\R_{I}^{n}]) $ for each $ n \in \N $.
\end{cor}
   Thus we can construct a theory on $ \R_{I}^{n} $ that completely paralleis that on $ \R^n $. 
     Since $ \R_{I}^{n} \subset \R_{I}^{n+1},$ we have an increasing sequence, so we define $$ \widehat{\R}_{I}^{\infty} = \lim\limits_{ n \to \infty} \R_{I}^{n} = \bigcup\limits_{k=1}^{\infty} \R_{I}^{k}.$$
     In \cite {GM} , it is shown that we can extend the measure $ \lambda_{\infty}(.) $ to $ \R^\infty $.\\
     We take measurable functions as follows.      Let $ x =(x_1,x_2,...) \in \R_{I}^{\infty},$  $ I_n = \Pi_{k=n+1}^{\infty}[\frac{-1}{2}, \frac{1}{2}] $ and let $ h_n(\widehat{x})= \chi_{I_n}(\widehat{x}),$ where $\widehat{x} = (x_i)_{i=n+1}^{\infty}.$

  \begin{Def}
    Let $ M^n $ represents the class of Lebesgue measurable functions on $ \R^n$. If $ x \in \R_{I}^{\infty} $ and $ f^n \in M^n,$ let $ \widehat{x}= (x_i)_{i=1}^{n} $ be a  measurable function of order $ n $ (or $e_n$-tame) on $ \R_{I}^{\infty} $ by $ f(x) = f^n(\overline{x})\otimes h_n(\widehat{x}).$ We let $$ M_{I}^{n} = \{ f(x)~: f(x) =  f^n(\overline{x})\otimes h_n(\widehat{x})~, x \in \R_{I}^{\infty} \} $$ be the class of all $e_n-$tame functions.
     \end{Def} 
    \begin{Def}
    A function $ f: \R_{I}^{\infty} \to \R $ is said to be measurable and we write $ f \in M_I,$ if there is a sequence $\{ f_n \in M_{I}^{n} \}$ of $ e_n$-tame functions, such that $ \lim\limits_{n \to \infty}f_n(x) \to f(x)~\lambda_{\infty}$-a.e.\\ As an application of $\R_I^\infty,$  we can construct $\mcB_J^n,$ where $\mcB$ is separable Banach space.
     Let $\{e_k\} $ be an S-basis for $\mcB,$ and let $x=\sum\limits_{n=1}^{\infty}x_n e_n,$ and from $\mathcal{P}_n(x)= \sum\limits_{k=1}^{n}x_ke_k $ and $\mathcal{Q}_nx=(x_1, x_2,...,x_n) $ then  $\mcB_J^n$ defined $$\mcB_J^n=\{\mathcal{Q}_n x : x \in \mcB\} \times J^n $$ with norm $$||(x_k)||_{\mcB_J^n}= \max_{1 \leq k \leq n}||\sum_{i=1}^{k}x_ie_i||_{\mcB}=\max_{1 \leq k \leq n} ||\mathcal{P}(x)||_\mcB.$$
  As $\mcB_J^n \subset \mcB_J^{n+1} $ so we can set $\mcB_J^\infty = \bigcup\limits_{n=1}^{\infty}\mcB_J^n $ and $\mcB_J $ is a subset of $\mcB_J^\infty$. \\
  We set $\mcB_J $ as $\mcB_J= \{(x_1, x_2,...): \sum\limits_{k=1}^{\infty}x_ke_k \in \mcB\}$ and norm on $\mcB_J $ by $$||x||_{\mcB_J}= \sup_{n}||\mathcal{P}_n(x)||_\mcB=||x||_\mcB.$$
   If we consider $\mcB[\mcB_J^\infty] $ be the smallest $\sigma$-algebra containing $\mcB_J^\infty$ and define $\mcB[\mcB_J]= \mcB[\mcB_J^\infty] \cap \mcB_J$ then by a known result 
   \begin{equation}\label{eqa1}
   ||x||_\mcB= \sup_{n}||\sum_{k=1}^{n}x_ke_k||_\mcB
   \end{equation}
   is an equivalent norm on $\mcB$.
   \begin{prop}
    When $\mcB $ carries the equivalent norm (\ref{eqa1}), the operator $$T:(\mcB, ||.||_\mcB) \to (\mcB_J, ||.||_{\mcB_J} $$ defined by $T(x)=(x_k) $ is an isometric isomorphism from $\mcB$ onto $\mcB_J$.
    \end{prop}
    This shows that every Banach space with an S-basis has a natural embedding in $\R_I ^\infty.$ 
     So we call $\mcB_J $ the canonical representation of $\mcB $ in $\R_I^\infty.$
     With $\mcB[\mcB_J]= \mcB_J \cap \mcB[\mcB_J^\infty] $ we define $\sigma-$algebra generated by $\mcB$ and associated with $\mcB[\mcB_J] $ by $$\mcB_J[\mcB]=\{T^{-1}(A)~~| A \in \mcB[\mcB_J]\} = T^{-1}\{\mcB[\mcB_J]\}$$ since $\lambda_\infty(A_J^n)=0 $ for $A_J^n \in \mcB[\mcB_J^n]$ with $A_J^n$ compact, we see $\lambda_\infty(\mcB_J^n)=0,~n \in \N$. So, $\lambda_\infty(\mcB_J)=0 $ for every Banach space with an S-basis. Thus the restriction of $\lambda_\infty $ to $\mcB_J $ will not induce a non trivial measure on $\mcB.$\\
          
     The definition $1.5$  highlights our requirement that all functions on infinite dimensional space must be constructively defined as (essentially) finite dimensional limits.    
        \end{Def} 
         Using the definition $ \mcZ^p[\R_{I}^{n}] $ we see that $ \mcZ^p[\R_{I}^{n}] \subset \mcZ^p[\R_{I}^{n+1}].$\\
         Thus we can define $$ \mcZ^p[\widehat{\R}_{I}^{\infty}]= \bigcup\limits_{n=1}^{\infty} \mcZ^p[\R_{I}^{n}]$$
         \begin{Def}
       We say that, for $ 1 \leq p \leq \infty$ a measurable function $ f \in \mcZ^p[\R_{I}^{\infty}] $ if there is a Cauchy sequence $\{f_n\} \subset \mcZ^p[\widehat{\R}_{I}^{\infty}] $ with $ f_n \in \mcZ^p[\R_{I}^{n}] $ and $\lim\limits_{n \to \infty}f_n(x) = f(x)~.\lambda_{\infty}$-a.e.
       \end{Def}
      
       \begin{Def}
        Using  definition of $ KS^p[\mathbb{R}_{I}^{n}] $ we see that $ KS^p[\mathbb{R}_{I}^{n}] \subset KS^p[\mathbb{R}_{I}^{n+1}].$ 
         Thus we can define $$ KS^p[\widehat{\mathbb{R}}_{I}^{\infty}]= \bigcup\limits_{n=1}^{\infty} KS^p[\mathbb{R}_{I}^{n}].$$
       \end{Def}
       
       \begin{Def}
     If $ f \in \mcZ^p[\R_{I}^{\infty}],$ we define the integral of $ f$ by $$ \int_{\R_{I}^{\infty}} f(x)d\lambda_{\infty}(x) = \lim\limits_{n \to \infty}\int_{\R_{I}^{n}} f_n(x)d\lambda_{\infty}(x)$$ where $ f_n \in \mcZ^p[\R_{I}^{n}]$ for all $n$  and the family $\{f_n \}$ is a Cauchy sequence.
     \end{Def}
     
 \begin{thm}
 If $ f \in \mcZ^p[\mcB] $ then the integral of $f $ by $$\int_{\mathcal{B}}f(y)d\lambda_{\mathcal{B}}(y) = \lim\limits_{ n \to \infty}\int_{R_{I}^{n}}f_n(x)d\lambda_{\infty}(x)$$ where $f_n \subset \mcZ^p[\widehat{\mcB}] $ for all $ n $ and the family $\{f_n \} $ is a Cauchy sequence is exists and is unique for every $ f \in \mcZ^p[\mcB].$
 \end{thm}

     The purpose of this paper is to form Zachary space in $\R_I^\infty$, we want to find our space contains BMO$[\R_I^\infty]$ as dense continuous embedding. For application of $\mcZ^p[\R_I^\infty]$, we want to developed $\mcZ^p[\mcB]$, where $\mcB$ is separable Banach space. 
     \section{Space of functions of bounded mean oscillation (BMO$[\R_I^\infty]$)}
     The space of functions of bounded mean oscillation or BMO naturally arises as the class of functions whose deviation from their mean over cubes is bounded. The space BMO shares similar properties with the space $L^\infty$. In this section we mention few properties of BMO$[\R_I^\infty]$, proof are similar as \cite{GRA}.
     \begin{Def}
     \begin{enumerate}
     \item Let $f \in L_{loc}^1[\R_I^\infty] $ and $Q$ be a cube in $\R_I^\infty$. We define the average of $f $ over $Q$ by  avg $f_{Q}= \frac{1}{\lambda_{\infty}(Q)}\int_Qf(y)d\lambda_{\infty}(y)$
     \item We define the sharp maximal function $\M^{\#}(f)(x)=\sup\limits_{Q}\frac{1}{\lambda_{\infty}(Q)}\int_{Q}|f(y)-Avg_Q|d\lambda_{\infty}(y).$
     \end{enumerate}
     \end{Def}
     If $\M^{\#}(f)(x) \in L^\infty[\R_I^\infty]$, we say that $f $ is a bounded mean oscillation. More precisely, the space of functions of bounded mean oscillation defined by $$ BMO[R_I^\infty]=\{f \in L_{loc}^1[\R_I^\infty]~~: \M^{\#}(f) \in L^\infty[\R_I^\infty]\} $$  and $$ ||f||_{BMO}=||\M^{\#}(f)||_{L^\infty}.$$
     We can also define $BMO[\R_I^\infty]$ as follows. 
     As $BMO[\R_I^n] \subset BMO[\R_I^{n+1}],~$ with $BMO[\widehat{\R}_I^\infty]=\bigcup\limits_{k=1}^{\infty}BMO[\R_I^n].$
     
     \begin{thm}
     BMO$[\R_I^\infty]$ is a linear space.
     \end{thm}
     \begin{proof}
     Let $ f,~g \in BMO[\R_I^\infty].$      Then
     \begin{align*}
     ||f+g||_{BMO} &=\sup\limits_{Q\subset \R_I^\infty}\frac{1}{\lambda_\infty(Q)}\int_{Q}|(f+g)(x)-Avg_{Q}(f+g)|d\lambda_\infty\\&=\sup\limits_{Q\subset \R_I^\infty}\frac{1}{\lambda_\infty(Q)}\int_{Q}|f(x)+g)(x)-Avg_{Q}f-Avg_{Q}g|d\lambda_\infty\\& \leq ||f||_{BMO}+||g||_{BMO}< \infty.
     \end{align*}
     Therefore, $ f+g \in BMO[\R_I^\infty].$\\
     Again for $\alpha \in \C$,
     \begin{align*}
     ||\alpha f||_{BMO}&=|\alpha|||f||_{BMO}< \infty.
          \end{align*}
     Thus $\alpha f \in BMO[\R_I^\infty].$ Hence the result.
      \end{proof}
      \begin{thm}\label{th23}
      \begin{enumerate}
      \item If $||f||_{BMO}=0$ then $ f $ is a.e. equal to a constant.
      \item $L^\infty[\R_I^\infty]  \subset BMO[\R_I^\infty].$
      \end{enumerate}
      \end{thm}
      \begin{proof}
      $(1)$ If $||f||_{BMO[\R_I^\infty]}=0$, then $$ \int_{Q \subset \R_I^\infty}|f(x)-Avg_{Q}f|d\lambda_\infty=0.$$ Then $f $ has to be a.e. equal to its average $A_N$ over every cube $$Q=[-N, N]^n \subset \R_I^\infty.$$ As $[-N, N]^n \subset [-N-1, N+1]^n$, it follows $A_N= A_{N+1}. $\\
      $(2) $ Let $ f \in L^\infty[\R_I^\infty].$       Therefore, 
      \begin{align*}
      Avg_Q|f(x)-Avg_{Q}f| &\leq 2 Avg_{Q}|f|\\&\leq 2||f||_{L^\infty}.
      \end{align*}
      Therefore, $f \in BMO[\R_I^\infty].$
      \end{proof}
      \begin{thm}
      \begin{enumerate}
      \item  Suppose that there exists an $A>0$ such that for all cubes $Q \in \R_I^\infty$ there exists a constant $C_{Q} $ such that $$\sup\limits_{Q}\frac{1}{\lambda_\infty(Q)}\int_{Q}|f(x)-C_{Q}|d\lambda_\infty(x) \leq A.$$ Then $ f \in BMO[\R_I^\infty]$ and $||f||_{BMO[\R_I^\infty]} \leq 2A.$
      \item For all $ f $ locally integrable in $\R_I^\infty$, we have $$\frac{1}{2}||f||_{BMO[\R_I^\infty]} \leq \sup\limits_{Q}\frac{1}{\lambda_\infty(Q)}\inf\limits_{C_Q}\int_{Q}|f(x)-C_Q|d\lambda_\infty(x) \leq ||f||_{BMO[\R_I^\infty]}.$$
      \item If $ f \in BMO[\R_I^\infty],~h \in \R_I^\infty$ and $\tau^h(f) $ is given by $\tau^h(f)(x)=f(x-h)$ then $\tau^h(f) $ is also in $BMO[\R_I^\infty]$ and $||\tau^h(f)||_{BMO[\R_I^\infty]}=||f||_{BMO[\R_I^\infty]}.$
      \end{enumerate}
      \end{thm}
      \begin{proof}
      For proof of $(1)$, we use  \cite[Propositon 7.1.2, pp118]{GRA}.\\ 
      The proof of $(2)$ and $(3)$ follows from $(1)$.
      \end{proof}
    \begin{rem}
 If we consider functions that differ by a constant as equivalent then it is easy to see that $BMO[\R_I^\infty] $ is a Banach space.
 \end{rem}
 \begin{thm}
 The duality of $BMO[\R_I^\infty] $ is $H^1[\R_I^\infty]$, where $H^1$ is Hardy space.
 \end{thm}
 \begin{proof}
  As duality of $BMO[\R^n] $ is $H^1[\R^n]$, that is duality of $BMO[\R_I^n]$ is $H^1[\R_I^n]$. We can find duality of $\bigcup\limits_{n=1}^{\infty}BMO[\R_I^n]$ is $\bigcup_{n=1}^{\infty}H^1[\R_I^\infty].$\\
  So, duality of $BMO[\R_I^\infty]$ is $H^1[\R_I^\infty]$.
 \end{proof}

\section{Zachary space $\mcZ^p[\R_I^\infty]$}
 We obtain an equivalent definition of $BMO[\R_I^\infty]$ using balls. 
     Let $\{\mcE_k(x)\} $ be the family of generating functions for $K{S^2}[\R_I^\infty]$ and recalling that the characteristic functions for a family of cubes $\{Q_k\}$ centered at each rational point in $\R_I^\infty$. Let $f \in L_{loc}^1[\R_I^\infty]$ and define $f_{ak} $ by $$f_{ak}= \frac{1}{\lambda_\infty[Q_k]}\int_{Q_k}f(y)d\lambda_\infty(y)= \frac{1}{\lambda_\infty[Q_k]}\int_{Q_k}\mcE_k(y)f(y)d\lambda_\infty(y).$$
 If $p, ~~1 \leq p \leq \infty $ and $t_k=2^{-k}, $ we define $||f||_{Z^p}$ by 
  $$||f||_{Z^p[\R_I^\infty]}= \left\{\begin{array}{c}\left(\sum\limits_{k=1}^{\infty}t_k|\frac{1}{\lambda_{\infty}[\R_I^\infty]}\int_{\R_{I}^{\infty}}[f(y)-f_{ak}]d\lambda_{\infty}(x)|^p\right)^{\frac{1}{p}}, \mbox{~for~} 1\leq p<\infty;\\
  \sup\limits_{k\geq 1}|\frac{1}{\lambda_{\infty}[\R_I^\infty]}\int_{\R_{I}^{\infty}}[f(y)-f_{ak}]d\lambda_{\infty}(x)|, \mbox{~for~} p=\infty, \end{array}\right.$$ where $x=(x_1,x_2,x_3,....) \in R_{I}^{\infty}.$\\
   With $\R_{I}^{\infty}= \lim\limits_{n \to \infty}\R_{I}^{n}= \bigcup\limits_{n=1}^{\infty}\R_{I}^{k}$ and $   \R_{I}^{n}= \R^n \times I_n,$ where $I=[\frac{-1}{2}, \frac{1}{2}] .$\\
   The set of functions for which $||f||_{\mcZ^p}< \infty $ is called the Zachary functions of bounded mean oscillation and order $p,~~1 \leq p \leq \infty$.
    \begin{thm}
       Uniqueness and existence: For $ 1 \leq p \leq \infty$ a measurable function $ f \in \mcZ^p[\R_{I}^{\infty}] $ if there is a Cauchy sequence $\{f_n\} \subset \mcZ^p[\widehat{\R}_{I}^{\infty}] $ with $ f_n \in \mcZ^p[\R_{I}^{n}] $ and $\lim\limits_{n \to \infty}f_n(x) = f(x) ~~\lambda_{\infty}$-a.e.
       \end{thm}
       \begin{proof}
       Since the family of functions $\{f_n\}$ is Cauchy, it is follows that if the integral exists, it is unique.  To prove existence, follow the standard argument and first assume that $f(x) \ge 0$.  In this case, the sequence can always be chosen to be increasing, so that the integral exists.  The general case now follows by the standard decomposition.
       \end{proof}
   \begin{thm}
   If $\mcZ^p[\R_I^\infty] $ is the class of Zachary functions of bounded mean oscillation and order $p, 1\leq p \leq \infty$, then $Z^p[\R_I^\infty] $ is linear. 
   \end{thm}
   \begin{proof}
   Let $f, g \in \mcZ^p[\R_I^\infty].$    Then for $ p < \infty$, 
   \begin{align*}
   ||f+g||_{\mcZ^p[\R_I^\infty]}&= \left(\sum\limits_{k=1}^{\infty}t_k|\frac{1}{\lambda_{\infty}[\R_I^\infty]}\int_{\R_{I}^{\infty}}[(f+g)(y)-(f+g)_{ak}]d\lambda_{\infty}(x)|^p\right)^{\frac{1}{p}}\\&\leq ||f||_{\mcZ^p[\R_I^\infty]}+||g||_{\mcZ^p[\R_I^\infty}< \infty,
   \end{align*}
   so, $f +g \in \mcZ^p[\R_I^\infty].$\\
   Again let $\alpha \in \C$ and $f \in \mcZ^p[\R_I^\infty]$,
   \begin{align*}
   ||\alpha f||_{\mcZ^p[\R_I^\infty]}&= \left(\sum\limits_{k=1}^{\infty}t_k|\frac{1}{\lambda_{\infty}[\R_I^\infty]}\int_{\R_{I}^{\infty}}[(\alpha f)(y)-(\alpha f)_{ak}]d\lambda_{\infty}(x)|^p\right)^{\frac{1}{p}}\\&\leq |\alpha|||f||_{\mcZ^p[\R_I^\infty]}< \infty.
   \end{align*}
   Therefore, $ \alpha f \in \mcZ^p[\R_I^\infty].$\\ 
   For $p = \infty$, it is obvious.
   Hence the result.
   \end{proof}
\begin{thm}\label{th33}
\begin{enumerate}
\item $||\lambda f||_{\mcZ^p} \leq |\lambda||f||_{\mcZ^p}$
\item $||f +g||_{\mcZ^p} \leq||f||_{\mcZ_p} +||g||_{\mcZ_p}$
\item $||f||_{\mcZ_p}=0 $ {\rm implies} $ f=$ {\rm constant} (a.s).  
  \end{enumerate}
\end{thm}
\begin{proof}
$(1) $ and $(2) $ are easy, so we omit the proof.\\
$(3) $ Let $||\mcZ^p[\R_I^\infty]||=0$, implies $$|\frac{1}{\lambda_\infty[\R_I^\infty]}\int_{\R_I^\infty}[f(y)-f_{ak}]d\lambda_\infty(x)|^p=0.$$ This means $ f $ has to be a.s. equal to $f_{ak} $ over every balls in $\R_I^\infty.$ 
Remaining fact follows from Theorem \ref{th23}(1).
\end{proof}
\begin{rem}
 If we consider functions that differ by a constant as equivalent then it is easy to see that $\mcZ^p[\R_I^\infty] $ is a Banach space and $\mcZ^2[\R_I^\infty]$ is a Hilbert space.
 \end{rem}
   \begin{thm}
   \begin{enumerate}
   \item The space $\mcZ^\infty[\R_I^\infty]\subset \mcZ^p[\R_I^\infty],~~1\leq p \leq \infty,$ as a dense continuous embedding.
\item The space $BMO[\R_I^\infty]\subset \mcZ^\infty[\R_I^\infty],$ as a continuous embedding.
\end{enumerate}
\end{thm}
\begin{proof}
(1) Let $ f \in \mcZ^\infty[\R_I^\infty]$. This implies $$|\frac{1}{\lambda_\infty[\R_I^\infty]}\int_{\R_I^\infty}[f(y)-f_{ak}]d\lambda_\infty(x)|$$ is uniformly bounded for all $k.$ It follows that $|\frac{1}{\lambda_\infty[\R_I^\infty]}\int_{\R_I^\infty}[f(y)-f_{ak}]d\lambda_\infty(x)|^p$ is uniformly bounded for all $p,~~1 \leq p < \infty$. \\
As $\sum\limits_{k=1}^{\infty}t_k=1$, from the definition of $\mcZ^p[\R_I^\infty]$, we get $$\{\sum\limits_{k=1}^{\infty}t_k|\frac{1}{\lambda_\infty[\R_I^\infty]}\int_{\R_I^\infty}[f(y)-f_{ak}]d\lambda_\infty(x)|^p\}^{\frac{1}{p}}< \infty$$. Therefore $ f \in \mcZ^p[\R_I^\infty].$\\
(2) As $BMO[\R^n] \subset \mcZ^\infty[\R^n]$ as a dense continuous embedding.
That is $BMO[\R_I^n] \subset \mcZ^\infty[\R_I^n]$ as a dense continuous embedding. However $\mcZ^\infty[\R_I^\infty]$ is the closure of $\bigcup\limits_{n=1}^{\infty}\mcZ^\infty[\R_I^n]$. It follows $\mcZ^\infty[\R_I^\infty]$ contains $\bigcup\limits_{n=1}^{\infty}BMO[\R_I^n]$ which is dense in $BMO[\R_I^\infty]$ as its closure.
\end{proof}
\section{Zachary space $\mcZ^p[\mcB]$, where $\mcB$ is separable Banach space}
Let $\{\mcE_k(x)\} $ be the family of generating functions for $K{S^2}[\mcB]$ and recalling that the characteristic functions for a family of cubes $\{Q_k\}$ centered at each rational point in $\mcB$. Let $f \in L_{loc}^1[\mcB]$ and define $f_{ak} $ by $$f_{ak}= \frac{1}{\lambda_\mcB[Q_k]}\int_{Q_k}f(y)d\lambda_\mcB(y)= \frac{1}{\lambda_\mcB[Q_k]}\int_{Q_k}\mcE_k(y)f(y)d\lambda_\mcB(y).$$
 If $p,~1 \leq p \leq \infty $ and $t_k=2^{-k}, $ we define $||f||_{\mcZ^p}$ by 
  $$||f||_{\mcZ^p[\mcB]}= \left\{\begin{array}{c}\left(\sum\limits_{k=1}^{\infty}t_k|\frac{1}{\lambda_{\mcB}[\mcB]}\int_{\mcB}[f(y)-f_{ak}]d\lambda_{\mcB}(x)|^p\right)^{\frac{1}{p}}, \mbox{~for~} 1\leq p<\infty;\\
  \sup_{k\geq 1}|\frac{1}{\lambda_{\mcB}[\R_I^\infty]}\int_{\mcB}[f(y)-f_{ak}]d\lambda_{\mcB}(x)|, \mbox{~for~} p=\infty \end{array}\right.$$ 
   The set of functions for which $||f||_{\mcZ^p}< \infty $ is called the Zachary functions of bounded mean oscillation and order $p,~1 \leq p \leq \infty$.\\
   We can also construct $\mcZ^p[\mcB]$ as follows. 
Let $\mcB$ be separable Banach space with $S-$ basis.
Let $\mcZ^p[\widehat{\mcB}]= \bigcup\limits_{k=1}^{\infty}\mcZ^p[\mcB^k].$
    \begin{thm}
       Uniqueness and existence:  If $ f \in \mcZ^p[\mcB] $ then the integral of $f $ by $$\int_{\mathcal{B}}f(y)d\lambda_{\mathcal{B}}(y) = \lim\limits_{ n \to \infty}\int_{R_{I}^{n}}f_n(x)d\lambda_{\infty}(x),$$ where $f_n \subset \mcZ^p[\widehat{\mcB}] $ for all $ n $ and the family $\{f_n \} $ is a Cauchy sequence is exists and is unique for every $ f \in \mcZ^p[\mcB].$
       \end{thm}
       \begin{proof}
       Since the family of functions $\{f_n\}$ is Cauchy, it is follows that if the integral exists, it is unique.  To prove existence, follow the standard argument and first assume that $f(x) \ge 0$.  In this case, the sequence can always be chosen to be increasing, so that the integral exists.  The general case now follows by the standard decomposition.
       \end{proof}
   \begin{thm}
   If $\mcZ^p[\mcB] $ is the class of Zachary functions of bounded mean oscillation and order $p, 1\leq p \leq \infty$, then $\mcZ^p[\mcB] $ is linear.
   \end{thm}
   \begin{proof}
   Let $f, g \in \mcZ^p[\mcB].$
   Then for $ p < \infty$, 
   \begin{align*}
   ||f+g||_{\mcZ^p[\mcB]}&= \left(\sum\limits_{k=1}^{\infty}t_k|\frac{1}{\lambda_{\mcB}[\mcB]}\int_{\mcB}[(f+g)(y)-(f+g)_{ak}]d\lambda_{\mcB}(x)|^p\right)^{\frac{1}{p}}\\&\leq ||f||_{\mcZ^p[\R_I^\mcB]}+||g||_{\mcZ^p[\mcB]}\\&< \infty.
   \end{align*}
   so, $f +g \in \mcZ^p[\mcB].$\\
   Again let $\alpha \in \C$ and $f \in \mcZ^p[\mcB]$,
   \begin{align*}
   ||\alpha f||_{\mcZ^p[\mcB]}&= \left(\sum\limits_{k=1}^{\infty}t_k|\frac{1}{\lambda_{\mcB}[\mcB]}\int_{\mcB}[(\alpha f)(y)-(\alpha f)_{ak}]d\lambda_{\mcB}(x)|^p\right)^{\frac{1}{p}}\\&\leq |\alpha|||f||_{\mcZ^p[\mcB]}\\&< \infty.
   \end{align*}
   Therefore, $ \alpha f \in \mcZ^p[\mcB].$\\ For $p = \infty$, it is obvious.
   Hence the result.
   \end{proof}
\begin{thm}
\begin{enumerate}
\item $||\lambda f||_{\mcZ^p} \leq |\lambda||f||_{\mcZ^p}$
\item $||f +g||_{\mcZ^p} \leq||f||_{\mcZ_p} +||g||_{\mcZ_p}$
\item $||f||_{\mcZ_p}=0 $ {\rm implies} $ f=$ {\rm constant} (a.s).
  
  \end{enumerate}
\end{thm}
\begin{proof}
$(1) $ and $(2) $ are easy, so we omit the proof here.\\
 $(3) $ Let $||\mcZ^p[\mcB]=0$, implies $$|\frac{1}{\lambda_\mcB[\mcB]}\int_{\mcB}[f(y)-f_{ak}]d\lambda_\mcB(x)|^p=0.$$ This means $ f $ has to be a.s. equal to $f_{ak} $ over every balls in $\mcB.$\\
Remaining fact follows from Theorem \ref{th33} (1).
\end{proof}
\begin{rem}
 If we consider functions that differ by a constant as equivalent then it is easy to see that $\mcZ^p[\mcB] $ is a Banach space and $\mcZ^2[\mcB]$ is a Hilbert space.
 \end{rem}
   \begin{thm}
   \begin{enumerate}
   \item The space $\mcZ^\infty[\mcB]\subset \mcZ^p[\mcB],~1\leq p \leq \infty,$ as a dense continuous embedding.
\item The space $BMO[\mcB]\subset \mcZ^\infty[\mcB],$ as a continuous embedding.
\end{enumerate}
\end{thm}
\begin{proof}
(1) Let $ f \in \mcZ^\infty[\mcB]$. This implies $$|\frac{1}{\lambda_\mcB[\mcB]}\int_{\mcB}[f(y)-f_{ak}]d\lambda_\mcB(x)|$$ is uniformly bounded for all $k.$ It follows that $|\frac{1}{\lambda_\mcB[\mcB]}\int_{\mcB}[f(y)-f_{ak}]d\lambda_\mcB(x)|^p$ is uniformly bounded for all $p,~~1 \leq p < \infty$. \\
As $\sum\limits_{k=1}^{\infty}t_k=1$, from the definition of $\mcZ^p[\mcB]$, we get $$\{\sum\limits_{k=1}^{\infty}t_k|\frac{1}{\lambda_\mcB[\mcB]}\int_{\mcB}[f(y)-f_{ak}]d\lambda_\mcB(x)|^p\}^{\frac{1}{p}}< \infty$$. Therefore $ f \in \mcZ^p[\mcB].$\\
(2) As $BMO[\mcB^n] \subset \mcZ^\infty[\mcB^n]$ as a dense continuous embedding.
That is $BMO[\mcB_J^n] \subset \mcZ^\infty[\mcB_J^n]$ as a dense continuous embedding. However $\mcZ^\infty[\mcB_J^\infty]$ is the closure of $\bigcup\limits_{n=1}^{\infty}\mcZ^\infty[\mcB_J^n]$. It follows $\mcZ^\infty[\mcB]$ contains $\bigcup\limits_{n=1}^{\infty}BMO[\mcB_J^n]$ which is dense in $BMO[\mcB]$ as its closure.
\end{proof}

\begin{thm}
      $\mcZ^p[\mcB] $ and $\mcZ^p[\mcB_J]$ are equivalent spaces.
     \end{thm}
     \begin{proof}
     If $\mcB$ is a separable Banach space, $ T $ maps $\mcB$ onto $\mcB_J \subset R_I^\infty,$ where $ T $ is a isometric isomorphism so that $\mcB_J$ is a embedding of $\mcB $ into $\R_I^\infty.$ This is how we able to define a Lebesgue integral on $\mcB $ using $\mcB_J$ and $T^{-1}.$ Thus $\mcZ^p[\mcB] $ and $\mcZ^p[\mcB_J]$ are not different space.
     \end{proof}
\begin{thm}
     $\mcZ^p[\mcB] \subset \mcZ^p[\R_I^\infty]$ embedding as closed subspace.
     \end{thm}
     \begin{proof}
      As every separable Banach space can be embedded in $\R_I^\infty$ as a closed subspace containing $\mcB_J^\infty$,\\
      So, $\mcZ^p[\mcB_J^\infty] \subset \mcZ^p[\R_I^\infty] $ embedding  as a closed subspace. That is $\mcZ^p[\bigcup\limits_{n=1}^{\infty}\mcB_{J}^{n}] \subset \mcZ^p[\R_I^\infty] $ embedding as a closed subspace.\\
      So, $\mcZ^p[\mcB_J^n] \subset \mcZ^p[\R_I^\infty] $ embedding as a closed subspace. Finally we can conclude that $\mcZ^p[\mcB] \subset \mcZ^p[\R_I^\infty]$ embedding as closed subspace.
     \end{proof}
 \section{Acknowledgement: } The authors would like to thank Prof. Tepper L. Gill for suggesting the formation of this new space  and
 making valuable suggestions  that improve the presentation of the paper.
 \section{Declaration}
 \noindent {\bf Funding:} Not Applicable, the research is not supported by any funding agency.\\
 {\bf Conflict of Interest/Competing interests:} The authors declare that the article is free from conflicts of interest.\\
 {\bf Availability of data and material:} The article does not contain any data for
 analysis.\\
 {\bf Code Availability:} Not Applicable.\\
 {\bf Author's Contributions:} All the authors have equal contribution for the preparation of the article.
 \bibliographystyle{amsalpha}

\end{document}